\documentclass[12pt]{amsart}
\usepackage{amsmath,amssymb,amsbsy,amsfonts,latexsym,amsopn,amstext,
                                               amsxtra,euscript,amscd,bm}
                    
\usepackage{url}
\usepackage[colorlinks,linkcolor=blue,anchorcolor=blue,citecolor=blue]{hyperref}
\usepackage{cleveref}
\usepackage{nameref}
\usepackage{color}
\numberwithin{equation}{section}
\usepackage[left=70pt, right=70pt]{geometry}
\usepackage{mathtools}

\begin{document}

\newtheorem{problem}{Problem}
\newtheorem{theorem}{Theorem}[section]
\newtheorem{lemma}[theorem]{Lemma}
\newtheorem{corollary}[theorem]{Corollary}
\newtheorem{conjecture}[theorem]{Conjecture}
\newtheorem{example}[theorem]{Example}
\newtheorem{claim}[theorem]{Claim}
\newtheorem{cor}[theorem]{Corollary}
\newtheorem{condition}[theorem]{Condition}
\newtheorem{prop}[theorem]{Proposition}
\newtheorem{definition}[theorem]{Definition}
\newtheorem{question}[theorem]{Question}
\newtheorem{conj}{Conjecture}
\newtheorem{hypothesis}{Hypothesis}
\def\cA{{\mathcal A}}
\def\cB{{\mathcal B}}
\def\cC{{\mathcal C}}
\def\cD{{\mathcal D}}
\def\cE{{\mathcal E}}
\def\cF{{\mathcal F}}
\def\cG{{\mathcal G}}
\def\cH{{\mathcal H}}
\def\cI{{\mathcal I}}
\def\cJ{{\mathcal J}}
\def\cK{{\mathcal K}}
\def\cL{{\mathcal L}}
\def\cM{{\mathcal M}}
\def\cN{{\mathcal N}}
\def\cO{{\mathcal O}}
\def\cP{{\mathcal P}}
\def\cQ{{\mathcal Q}}
\def\cR{{\mathcal R}}
\def\cS{{\mathcal S}}
\def\cT{{\mathcal T}}
\def\cU{{\mathcal U}}
\def\cV{{\mathcal V}}
\def\cW{{\mathcal W}}
\def\cX{{\mathcal X}}
\def\cY{{\mathcal Y}}
\def\cZ{{\mathcal Z}}

\def\A{{\mathbb A}}
\def\B{{\mathbb B}}
\def\C{{\mathbb C}}
\def\D{{\mathbb D}}
\def\E{{\mathbb E}}
\def\F{{\mathbb F}}
\def\G{{\mathbb G}}
\def\I{{\mathbb I}}
\def\J{{\mathbb J}}
\def\K{{\mathbb K}}
\def\L{{\mathbb L}}
\def\M{{\mathbb M}}
\def\N{{\mathbb N}}
\def\O{{\mathbb O}}
\def\P{{\mathbb P}}
\def\Q{{\mathbb Q}}
\def\R{{\mathbb R}}
\def\S{{\mathbb S}}
\def\T{{\mathbb T}}
\def\U{{\mathbb U}}
\def\V{{\mathbb V}}
\def\W{{\mathbb W}}
\def\X{{\mathbb X}}
\def\Y{{\mathbb Y}}
\def\Z{{\mathbb Z}}

\def\e{{\mathbf{e}}}
\def\ep{{\mathbf{e}}_p}
\def\eq{{\mathbf{e}}_q}

\DeclarePairedDelimiter\ceil{\lceil}{\rceil}
\DeclarePairedDelimiter\floor{\lfloor}{\rfloor}

\def\scr{\scriptstyle}
\def\({\left(}
\def\){\right)}
\def\[{\left[}
\def\]{\right]}
\def\<{\langle}
\def\>{\rangle}
\def\fl#1{\left\lfloor#1\right\rfloor}
\def\rf#1{\left\lceil#1\right\rceil}
\def\le{\leqslant}
\def\ge{\geqslant}
\def\eps{\varepsilon}
\def\mand{\qquad\mbox{and}\qquad}

\def\sssum{\mathop{\sum\ \sum\ \sum}}
\def\ssum{\mathop{\sum\, \sum}}
\def\ssumw{\mathop{\sum\qquad \sum}}

\def\vec#1{\mathbf{#1}}
\def\inv#1{\overline{#1}}
\def\num#1{\mathrm{num}(#1)}
\def\dist{\mathrm{dist}}

\def\fA{{\mathfrak A}}
\def\fB{{\mathfrak B}}
\def\fC{{\mathfrak C}}
\def\fU{{\mathfrak U}}
\def\fM{{\mathfrak M}}
\def\fm{{\mathfrak m}}
\def\fV{{\mathfrak V}}
\def\fI{{\mathfrak I}}

\newcommand{\bflambda}{{\boldsymbol{\lambda}}}
\newcommand{\bfxi}{{\boldsymbol{\xi}}}
\newcommand{\bfrho}{{\boldsymbol{\rho}}}
\newcommand{\bfnu}{{\boldsymbol{\nu}}}
\newcommand{\psum}{\sideset{}{^*}\sum}

\def\GL{\mathrm{GL}}
\def\SL{\mathrm{SL}}

\def\Hba{\overline{\cH}_{a,m}}
\def\Hta{\widetilde{\cH}_{a,m}}
\def\Hb1{\overline{\cH}_{m}}
\def\Ht1{\widetilde{\cH}_{m}}

\def\flp#1{{\left\langle#1\right\rangle}_p}
\def\flm#1{{\left\langle#1\right\rangle}_m}
\def\dmod#1#2{\left\|#1\right\|_{#2}}
\def\dmodq#1{\left\|#1\right\|_q}

\newcommand*\diff{\mathop{}\!\mathrm{d}}
\newcommand{\funccol}{\colon \thinspace}
\setlength\parindent{0pt}

\def\Zm{\Z/m\Z}

\def\Err{{\mathbf{E}}}

\newcommand{\comm}[1]{\marginpar{%
\vskip-\baselineskip 
\raggedright\footnotesize
\itshape\hrule\smallskip#1\par\smallskip\hrule}}

\def\xxx{\vskip5pt\hrule\vskip5pt}



\title{On the Goldbach problem with restricted primes}
\author{Michael Harm}
\address{School of Mathematics and Statistics, UNSW, 
Sydney, NSW 2052, Australia}
\email{m.harm@unsw.edu.au}

\begin{abstract}
    Let $N$ be a sufficiently large, odd integer. We prove an asymptotic formula for the number of representations of $N$ as the sum of three primes, one of which is smaller than a given $U$. By inserting the currently best zero-density estimate for Dirichlet $L$-functions, we may unconditionally take $U= N^{\frac{4}{49}}\exp(\log^{\frac{2}{3}+\varepsilon}N)$ for any $\varepsilon>0$. If we assume the Generalized Riemann Hypothesis instead, we may take $U= \log^{4+\varepsilon}N$.
\end{abstract}

\maketitle
\thispagestyle{empty}

\section{Introduction}
The binary Goldbach Conjecture is one of the most famous outstanding problems in number theory. It was first proposed during a letter exchange between Leonhard Euler and Christian Goldbach in the year 1742. It is characterized by
\begin{conjecture}\label{Conj: binary}
    Every even integer greater than $2$ is the sum of two primes.
\end{conjecture}
While Conjecture \ref{Conj: binary} remains unproven, there has been significant progress for related problems. Most prominent among these is the ternary Goldbach problem.
\begin{theorem}\label{thm: ternary}
    Every odd integer greater than $5$ is the sum of three primes.
\end{theorem}
Theorem \ref{thm: ternary} has first been proven asymptotically under the Riemann Hypothesis (RH) by Hardy and Littlewood in 1923 \cite{HL23}. Vinogradov then gave an unconditional asymptotic proof of the ternary Goldbach problem in 1937. And most recently it has been proven unconditionally by Helfgott in 2013 \cite{Hel13}. \\

A natural question to ask is how large gaps between Goldbach numbers (that is, even numbers that are the sum of two primes) can be. Assuming the Riemann Hypothesis (RH), Linnik \cite{Lin52} showed that for sufficiently large $N$, and any $\varepsilon>0$ there exists a Goldbach number in the interval $[N,N+\log^{3+\varepsilon}N]$. Linnik's result has been independently improved by Katai \cite{Kat67}, Montgomery-Vaughan \cite{MV75}, and Languasco-Perelli \cite{LP94}, who showed that, assuming RH, there exists a Goldbach number in any interval $[N,N+C\log^{2}N]$, for sufficiently large $N$ and a positive constant $C>0$.\\

Obviously Theorem \ref{thm: ternary} follows from Conjecture \ref{Conj: binary} if we add the prime $3$ to any even number greater than $2$. Another natural extension to gaps between Goldbach numbers is to ask if we can satisfy Theorem \ref{thm: ternary} with the additional condition of an upper bound $U$ on one of the primes. The first result of this type is due to Pan \cite{Pan59}, who showed $U=N^{\frac{1}{4}+\varepsilon}$ for any $\varepsilon>0$ and sufficiently large $N$. Zhan \cite{Zha95} obtained $U=N^{\frac{7}{72}+\varepsilon}$ using Huxley's \cite{Hux75} zero-density estimate. In the same paper Zhan also obtained $U=N^{\frac{7}{120}+\varepsilon}$ via sieve methods, these allow for much smaller bounds $U$, but forego an asymptotic formula. Wong \cite{Won96} improved this to $U=N^{\frac{7}{216}+\varepsilon}$ and Jia \cite{Jia96} to $U=N^{\frac{7}{240}+\varepsilon}$. Currently, the best result is by Cai \cite{Cai13}, who proved that asymptotically any odd integer $N$ is the sum of three primes, one of which is smaller than $U=N^{\frac{11}{400}+\varepsilon}$ for any $\varepsilon>0$. \\

The aim of this paper is to generalize and refine the methods of Zhan \cite{Zha95}, to obtain a direct correspondence between the ternary Goldbach problem with a small prime and zero-density estimates for Dirichlet $L$ functions, while also maintaining an asymptotic formula. Furthermore, we use the same methods to obtain a direct correspondence between the ternary Goldbach problem with a small prime and the Generalized Riemann Hypothesis (\nameref{con: GRH}).

\section{Main Results}
Let $N\in\N$, and let $U\leq y \leq \frac{N}{2}$. We define the unweighted and logarithmically weighted number of representations of the ternary Goldbach problem with restricted primes by
\begin{equation*}
\begin{split}
    \cR(N,y,U):=&\sum_{\substack{p_1+p_2+p_3=N\\N-2y< p_1\leq N-y\\y< p_2\leq 2y\\p_3\leq U}}1,\\
    R(N,y,U):=&\sum_{\substack{p_1+p_2+p_3=N\\N-2y< p_1\leq N-y\\y< p_2\leq 2y\\p_3\leq U}}\log p_1 \thinspace\log p_2\thinspace \log p_3.
\end{split}
\end{equation*}

\subsection{Results using zero-density estimates}
The results stated in this section are conditional on zero-density estimates of Dirichlet $L$-functions. Let $q\in\N$, let $\sigma \in \R$, let $T>0$ and let $L(s,\chi)$ be the Dirichlet $L$-function associated to the Dirichlet character $\chi$ modulo $q$. We denote by $\cN(\sigma,T,\chi)$ the number of zeroes $\rho=\beta +i\gamma$ of $L(s,\chi)$ with $\beta>\sigma$ and $|\gamma|\leq T$.
\begin{condition}\label{con: zero density}
    Let $q\in\N$, let $\frac{1}{2}\leq \sigma \leq 1$ and let $T>T_0>0$. There exist constants $A\geq 2$ and $B\geq 0$, such that
    \begin{equation*}
        \sum_{\chi\thinspace(q)}\cN(\sigma,T,\chi)\ll (qT)^{A(1-\sigma)}\log^B qT,
    \end{equation*}
    where the sum runs over all Dirichlet characters $\chi$ modulo $q$.
\end{condition}
Until recently the best zero-density estimate for Dirichlet $L$-functions could be attributed to Huxley with $A=\frac{12}{5}$ \cite[Equation 1.1]{Hux75}. However, a recent result by Guth-Maynard \cite{GM24} gives $A=\frac{30}{13}$ for $q=1$. Chen \cite{Che25} was able to improve Huxley's result using the methods of Guth-Maynard, per Chen's result we may take $A=\frac{7}{3}$ for all $q\in\N$. Barring any major roadblocks, it is feasible to expect Chen's result to be further improved to $A=\frac{30}{13}$ for any $q\in\N$. We call the case $A=2$ the Generalized Density Hypothesis, which is considered a weaker version of GRH.

\begin{theorem}\label{thm: main zero density}
    Assume Condition \ref{con: zero density}. Let $0<\varepsilon_1$ and let $0<\varepsilon_2<\frac{1}{3}$. For any odd $N\in\N$, $y\in [N^{1-\frac{1}{A}}\exp(\log^{\frac{2}{3}+\varepsilon_2}N),N/2]$ and $U\in[y^{1-\frac{1}{A}}\exp(\log^{\frac{2}{3}+\varepsilon_2}y),y/\log^{\varepsilon_1}N]$, we have
        \begin{equation*}
            \cR(N,y,U) \sim \frac{\mathfrak{S}(N)Uy}{\log U\log y\log N},
        \end{equation*}
    where $\mathfrak{S}$ denotes the infinite product
    \begin{equation*}
        \mathfrak{S}(N):=\prod_{p\nmid N}\bigg(1+ \frac{1}{(p-1)^3}\bigg)\prod_{p\mid N}\bigg(1- \frac{1}{(p-1)^2}\bigg).
    \end{equation*}
\end{theorem}
Note that $\mathfrak{S}(N)=0$ for even $N$, and $\mathfrak{S}(N)>\prod_{p>2}\big(1-\frac{1}{(p-1)^2}\big)> \frac{1}{2}$ for odd $N$.
\begin{corollary}\label{cor: main density 1}
    Assume Condition \ref{con: zero density}. Every sufficiently large odd integer $N$ is the sum of three primes
    \begin{equation*}
        N=p_1+p_2+p_3,
    \end{equation*}
    where $p_3\leq U=N^{(1-\frac{1}{A})(1-\frac{2}{A})}\exp(\log^{\frac{2}{3}+\varepsilon}N)$ for any $\varepsilon>0$.
\end{corollary}

\begin{corollary}\label{cor: main density 2}
    Assume Condition \ref{con: zero density}. Let $N\in\N$ be sufficiently large and let $U=N^{(1-\frac{1}{A})(1-\frac{2}{A})}\exp(\log^{\frac{2}{3}+\varepsilon}N)$ for any $\varepsilon>0$. There exists an integer $n\in [N,N+U]$ which is the sum of two primes
    \begin{equation*}
        N=p_1+p_2.
    \end{equation*}
\end{corollary}
Theorem \ref{thm: main zero density} and Corollaries \ref{cor: main density 1} \& \ref{cor: main density 2} are true for $A=\frac{7}{3}$ due to \cite{Che25}, i.e. we may take $U =N^\frac{4}{49}\exp(\log^{\frac{2}{3}+\varepsilon}N)$ for any $\varepsilon>0$. If the result of Guth-Maynard \cite{GM24} can be fully generalized to all $q\in \N$, we may take $A=\frac{30}{13}$, i.e. we may take $U =N^\frac{17}{225}\exp(\log^{\frac{2}{3}+\varepsilon}N)$. While our results unfortunately fall short of Cai's \cite{Cai13} $U=N^{\frac{11}{400}+\varepsilon}$ for now, our results are automatically strengthened by new and improved zero-density estimates for Dirichlet $L$-functions. For instance, a zero-density estimate with $A\approx 2.1103$ would be sufficient to improve on Cai's result. Furthermore, if we assume the Generalized Density Hypothesis (i.e. $A=2$) we may take $U=\exp(\log^{\frac{2}{3}+\varepsilon}N)$.
\subsection{Results under \nameref{con: GRH}}
The results in this section is conditional on the Generalized Riemann Hypothesis (\nameref{con: GRH}).
\begin{conjecture}[GRH]\label{con: GRH}
    For any Dirichlet character $\chi$ modulo $q\in\N$, we have
    \begin{equation*}
        \cN\bigg(\frac{1}{2},\infty,\chi\bigg)=0.
    \end{equation*}
\end{conjecture}
Using \nameref{con: GRH} we obtain a significantly stronger result than Theorem \ref{thm: main zero density}.
\begin{theorem}\label{thm: main GRH}
    Assume \nameref{con: GRH}. Let $0<\varepsilon$. Let $y\in [\sqrt{N}\log^{2+\varepsilon},N/2]$ and\\ \mbox{$U\in[\log^{4+\varepsilon}N,y/\log^{\varepsilon}N]$.} We have
    \begin{equation*}
        \cR(N,y,U) \asymp\frac{\mathfrak{S}(N)Uy}{\log U\log y\log N}. 
    \end{equation*}
\end{theorem}
Note again that $\mathfrak{S}(N)=0$ for even $N$, and $\mathfrak{S}(N)> \frac{1}{2}$ for odd $N$. Furthermore, we may turn this into an asymptotic formula by taking e.g. $U<p_3\leq 2U$ instead of $p_3\leq U$.
\begin{corollary}
    Assume Condition \ref{con: zero density}. Every sufficiently large odd integer $N$ is the sum of three primes
    \begin{equation*}
        N=p_1+p_2+p_3,
    \end{equation*}
    where $p_3\leq U=\log^{4+\varepsilon} N$ for any $\varepsilon>0$.
\end{corollary}

\section{Preliminary Results}

\subsection{Weighted representation functions}
The functions $\cR$ and $R$ are asymptotically related. We show their relation in this section.

\begin{lemma}\label{lemma: weighted reps}
    Let $\exp(\log^\varepsilon N)\leq U \leq y\leq \frac{N}{2}$ for any $0<\varepsilon<1$ and let $c>0$, we have
    \begin{equation*}\label{eq: R to mathcal R}
        \cR(N,y,U) = \bigg(1 + \cO\bigg( \frac{\log\log N}{\log^\varepsilon N}\bigg)\bigg) \frac{R_3(N,y,U)}{\log U \log y\log N}+\cO\bigg(\frac{Uy}{\log U\log y \log^{1+c} N}\bigg).
    \end{equation*}
    If instead we take $\log^\varepsilon N\leq U \leq y\leq \frac{N}{2}$ for any $1<1+c<\varepsilon$, we have
    \begin{equation*}
        \bigg|\cR(N,y,U)-\frac{R_3(N,y,U)}{\log U \log y\log N}\bigg| \leq \bigg(\frac{1+c}{\varepsilon} + \cO\bigg( \frac{\log\log y}{\log y}\bigg)\bigg) \frac{R_3(N,y,U)}{\log U \log y\log N}+\cO\bigg(\frac{Uy}{\log U\log y \log^{1+c} N}\bigg).
    \end{equation*}
\end{lemma}
\begin{proof}
    Let $C=\log^{1+c}N$ for any $c>0$. Since $\log p_2 =\log y + \cO(1)$ for $y<p_2\leq 2y$, we have
    \begin{equation*}
    \begin{split}
        \cR(N,y,U)=&\bigg(\frac{1}{\log y} +\cO\bigg(\frac{1}{\log^2 y}\bigg)\bigg)\sum_{\substack{p_1+p_2+p_3=N\\\max(N-2y,\frac{N}{C})<p_1\leq N-y\\y<p_2\leq 2y\\p_3\leq U}}\log p_2 \\
        &+\mathcal{O}\bigg(\sum_{p_3\leq U}1\sum_{\substack{N-2y<p_1\leq \max(N-2y,\frac{N}{C})}}1  \bigg).
    \end{split}
    \end{equation*}
    Since $N-2y\leq \frac{N}{C}$ implies $y\gg N$, the sum over $p_1$ in the error term is $\cO(\frac{y}{C\log y})$ due to the prime number theorem. Furthermore, for $N/C\leq n\leq N$ we have $\log n= \log N +\mathcal{O}(\log C)$, thus we obtain
    \begin{equation}\label{eq: weight 1}
    \begin{split}
        \cR(N,y,U)=&\bigg(\frac{1}{\log y} +\cO\bigg(\frac{1}{\log^2 y}\bigg)\bigg)\bigg(\frac{1}{\log N} +\cO\bigg(\frac{\log C}{\log^2 N}\bigg)\bigg)\sum_{\substack{p_1+p_2+p_3=N\\\max(N-2y,\frac{N}{C})<p_1\leq N-y\\y<p_2\leq 2y\\p_3\leq U}} \log p_1\log p_2\\
        &+\mathcal{O}\bigg(\frac{Uy}{\log U\log y\log^{1+c}N} \bigg)\\
        =&\bigg(\frac{1}{\log y} +\cO\bigg(\frac{1}{\log^2 y}\bigg)\bigg)\bigg(\frac{1}{\log N} +\cO\bigg(\frac{\log C}{\log^2 N}\bigg)\bigg)\sum_{\substack{p_1+p_2+p_3=N\\N-2y<p_1\leq N-y\\y<p_2\leq 2y\\p_3\leq U}} \log p_1\log p_2\\
        &+\mathcal{O}\bigg(\frac{Uy}{\log U\log y\log^{1+c}N} \bigg).
    \end{split}
    \end{equation}

    Similarly, for $\frac{U}{C}<p_3\leq U$, we have $|\log U -\log p_3|\leq \log C $, and therefore
    \begin{equation}\label{eq: weight 2}
    \begin{split}
        \bigg|\frac{R(N,y,U)}{\log U}- \sum_{\substack{p_1+p_2+p_3=N\\N-2y<p_1\leq N-y\\y<p_2\leq 2y\\p_3\leq U}} \log p_1\log p_2\bigg|\leq \frac{\log C}{\log^2 U} R(N,y,U) +\cO\bigg(\frac{Uy}{\log U\log^{c}N}\bigg).
    \end{split}
    \end{equation}
    We finish the proof by combining Equations \eqref{eq: weight 1} \& \eqref{eq: weight 2} and inserting the respective lower bounds for $U$.
\end{proof}
Note that Lemma \ref{lemma: weighted reps} could yield an asymptotic formula for all $U\geq \log^{\varepsilon}N$ with any $\varepsilon>0$, if one could show the upper bound
\begin{equation*}
    \sum_{\substack{p_1+p_2=N\\p_2\leq y}}1\ll\frac{y}{\log y\log N}.
\end{equation*}
Furthermore we immediately obtain an asymptotic formula for e.g. $\cR(N,y,U)-\cR(N,y,U/2)$ from Equation \eqref{eq: weight 1}.

\subsection{Circle Method}

In order to evaluate $R$, we apply the Hardy-Littlewood circle method. By making use of an orthogonality relation, we may write discrete sums as an integral over a unit interval $\R/\Z$. We may then evaluate said integral on subsets of the unit interval.

\begin{lemma}[Orthogonality relation]\label{lemma: orth relation}
    Let $k\in\Z$ and let $\mathfrak{U}=\R/\Z$ be any unit interval. We have
    \begin{equation*}
        \int_\mathfrak{U} \e(k)\diff\alpha=\begin{cases}
            1 & \text{if } k=0,\\
            0 & \text{otherwise,}
        \end{cases}
    \end{equation*}
    where we denote $\e_q(\alpha):=e^{2\pi i \alpha/q}$ for any $\alpha \in \R$ and $q\in\N$. We may omit the subscript for $q=1$.
\end{lemma}
\begin{proof}
    The proof is trivial and left as an exercise to the reader.
\end{proof}
We may now use Lemma \ref{lemma: orth relation} to write $R$ as an integral of generating functions.
\begin{lemma} \label{lemma: R to integral}
    Let $\fU=\R/\Z$ be any unit interval. We have
    \begin{equation*}\label{eq: RN}
        R(N,y,U)=\int_\fU S_1(\alpha)S_2(\alpha) S_3(\alpha) \e(-N\alpha) \diff \alpha,
    \end{equation*}
    where 
    \begin{equation*}
    \begin{split}
        S_1(\alpha):=&S_1(N,y,\alpha):=S(N-y,y,\alpha),\\
        S_2(\alpha):=&S_2(y,\alpha):=S(2y,y,\alpha),\\
        S_3(\alpha):=&S_3(U,\alpha):=S(U,U,\alpha),
    \end{split}
    \end{equation*}
    with
    \begin{equation*}\label{eq: generating function}
        S(x,h,\alpha):=\sum_{x-h<p\leq x}\log p\thinspace\e(p\alpha ).
    \end{equation*}
    
\end{lemma}
\begin{proof}
    We apply Lemma \ref{lemma: orth relation} to obtain
    \begin{equation*}\label{eq: ternary rep as int}
        \begin{split}
            R(N,y,U)=\sum_{\substack{N-2y<p_1\leq N-y\\ y<p_2\leq 2y\\p_3\leq U}}\log p_1\thinspace\log p_2 \thinspace\log p_3\int_\mathfrak{U} \e(\alpha(p_1+p_2+p_3-N))\diff\alpha.
        \end{split}
    \end{equation*}
    The Lemma follows immediately by exchanging the order of integration and summation.
\end{proof}

Hardy and Littlewood found that the generating function $S(x,h,\alpha)$ is large whenever $\alpha$ is close to a rational number $a/q$ with a small denominator $q$. Using Dirichlet's approximation theorem we may split up the unit interval $\fU$ into smaller intervals around rational points.

\begin{theorem}[Dirichlet, 1842]\label{thm: Dirichlet}
    Let $\alpha,Q\in\R$ with $Q>1$. There exist coprime integers $(a,q)=1$, such that $1\leq q\leq Q$ and 
    \begin{equation*}
        \bigg|\alpha -\frac{a}{q} \bigg|\leq \frac{1}{qQ}.
    \end{equation*}
\end{theorem}
\begin{proof}
    See e.g. \cite[Theorem 1A]{Sch80}.
\end{proof}

Let $2<Q:=\frac{U}{\log^{c_2}N}$ for some $c_2>0$ be a real, positive constant depending on $U$ and $N$. For every \mbox{$1\leq a\leq q \leq Q$} with $(a,q)=1$ we define the \textit{arcs}
\begin{equation*}
    \fM(q,a) := \bigg\{ \alpha\in\R : |\alpha-a/q|\leq \frac{1}{qQ} \bigg\}.
\end{equation*}
Theorem \ref{thm: Dirichlet} implies that the union of the arcs $\fM(q,a)$ is a covering of the unit interval
\begin{equation*}
    \fU :=(Q^{-1},1+Q^{-1}] \subset \bigcup_{q\leq Q}\bigcup_{\substack{a=1\\(a,q)=1}}^q \fM(q,a).
\end{equation*}
Let $1\leq P:=\log^{c_1}N\leq \frac{Q}{2}$ for some $0<c_1<c_2$. We canonically define the \textit{major arcs} $\fM$ by
\begin{equation*}
    \fM := \bigcup_{q\leq P}\bigcup_{\substack{a=1\\(a,q)=1}}^q \fM(q,a).
\end{equation*}

\begin{lemma}\label{lemma: disjoint major arcs}
    The arcs $\fM(q,a)\subset \fM$ with $1\leq a\leq q\leq P<\frac{Q}{2}$ and $(a,q)=1$ are disjoint and a subset of the unit interval $\fU$.
\end{lemma}
\begin{proof}
    Let $\frac{a}{q}\neq \frac{a'}{q'}$ and without loss of generality $q\leq q'$, we have
    \begin{equation*}
        \bigg| \frac{a}{q}-\frac{a'}{q'}\bigg|\geq \frac{1}{qq'}\geq \frac{2}{qQ}\geq \frac{1}{qQ}+\frac{1}{q'Q}.
    \end{equation*}
    Thus the major arcs are disjoint for $ q'\leq \frac{Q}{2}$. Furthermore we have $\fM\subset \fU$, since $\alpha> \frac{1}{Q}$ for all $\alpha\in\fM(\floor*{P},1)$ and $\alpha\leq 1+\frac{1}{Q}$ for all $\alpha\in\fM(1,1)$.
\end{proof}

We canonically define the \textit{minor arcs} by $\fm := \fU\setminus\fM$. Theorem \ref{thm: Dirichlet} gives us a covering of the minor arcs $\fm'\supseteq \fm$ defined by
\begin{equation*}
\begin{split}
    \fm'=  \bigcup_{P< q\leq Q}\bigcup_{\substack{a=1\\(a,q)=1}}^q \fM(q,a).
\end{split}
\end{equation*}

\begin{lemma}\label{lemma: R major minor split}
    We have
\begin{equation*}
\begin{split}
    R(N,y,U)= R_\fM +R_{\fm}
\end{split}
\end{equation*}
where
\begin{equation*}
    R_\fI:=R_\fI(N,y,U) :=\int_\mathfrak{I} S_1(\alpha)S_2(\alpha) S_3(\alpha) \e(-N\alpha)\diff \alpha,
\end{equation*}
for any $\mathfrak{I}\subseteq\mathfrak{U}$. 
\end{lemma}
\begin{proof}
    The proof follows immediately from Theorem \ref{thm: Dirichlet} and Lemmas \ref{lemma: R to integral} \& \ref{lemma: disjoint major arcs}. 
\end{proof}

\subsection{Input from the primes}
In this section we discuss the behaviour of the generating function $S(x,h,\alpha)$. For $\alpha:=\frac{a}{q}+\beta\in \fM(q,a)$, we define the error term of the generating function by
\begin{equation*}
    \Delta(x,h,\alpha):= S(x,h,\alpha)-\frac{\mu(q)}{\varphi(q)}v(x,h,\beta),
\end{equation*}
where 
\begin{equation*}
    v(x,h,\beta):=\sum_{x-h<n\leq x}\e(n\beta).
\end{equation*}
For $i\in\{1,2,3\}$, we also denote by $v_i(\beta)$ and $\Delta_i(\alpha)$ the geometric series and error term, respectively, associated to the generating function $S_i(\alpha)$ on the major arcs \mbox{$\alpha=\frac{a}{q}+\beta\in \fM(q,a)\subset \fM$.} For $i\in\{1,2,3\}$, the functions $v_i(\beta)$ and $\Delta_i(\alpha)$ are well-defined on the disjoint major arcs $\fM$ due to \mbox{Lemma \ref{lemma: disjoint major arcs}.} We may reduce $\Delta(x,h,\alpha)$ to the case $\Delta(x,h,\frac{a}{q})$ for $\alpha\in\fM(q,a)$ by the following Lemma.
\begin{lemma}\label{Lemma: Generating function}
    Let $\alpha=\beta+a/q\in\fM(q,a)$. We have
    \begin{equation*}
        |\Delta(x,h,\alpha)|\leq (1+|\beta|h) \max_{0<t\leq h} |\Delta(x,t,\frac{a}{q})|.
    \end{equation*}
\end{lemma}
\begin{proof}
    We write $\Delta(x,h,\frac{a}{q})=\sum_{x-h<n\leq x}c(n)$, with the coefficients
    \begin{equation*}
        c(n):=\log n\thinspace\e(an/q)-\frac{\mu(q)}{\varphi(q)},
    \end{equation*}
    if $n$ is a prime and $c(n):=-\frac{\mu(q)}{\varphi(q)}$ otherwise. By partial summation (see e.g. \mbox{\cite[Theorem 4.2]{Apo13})} we have
    \begin{equation*}
    \begin{split}
        \bigg|S(x,h,\alpha)-\frac{\mu(q)}{\varphi(q)}v(x,h,\beta)\bigg|=&\bigg|\sum_{x-h<n\leq x}c(n)\e(\beta n)\bigg| \\
        \leq & \bigg|\Delta(x,h,\frac{a}{q})\bigg|+ |\beta|\int_{x-h}^x\bigg|\Delta(x,t,\frac{a}{q})\bigg|\diff t\\
        \leq& (1+|\beta|h) \max_{0<t\leq h} |\Delta(x,t,\frac{a}{q})|.
    \end{split}
    \end{equation*}
\end{proof}

The following three Lemmas give estimates for the error term $\Delta$ at rational points.
\begin{lemma}\label{lemma: all unconditional}
    Let $0<\varepsilon_1<\varepsilon_2<1$ and let $q\leq \log^\ell x$ for any $\ell>0$. There exists a constant $c>0$ such that for any 
    \begin{equation*}
        h\in \bigg[x\exp\bigg(-\varepsilon_1\frac{c\log^\frac{3}{5}x}{(\log\log x)^\frac{1}{5}}\bigg),x\bigg],
    \end{equation*} we have
    \begin{equation*}
        \Delta(x,h,\frac{a}{q})\ll x\exp\bigg(-\varepsilon_2\frac{c\log^\frac{3}{5}x}{(\log\log x)^\frac{1}{5}}\bigg).
    \end{equation*}
    
\end{lemma}
\begin{proof}
    We apply \cite[Lemmas 3.3 \& 3.4]{Har26} to \cite[Corollary 2.2]{Har26}.
\end{proof}
\begin{lemma}\label{lemma: all zero density}
    Assume Condition \ref{con: zero density}. Let $0<\varepsilon_1<\varepsilon_2< 1/3$ and let $q\leq \log^\ell x$ for any $\ell>0$. For any $h\in [x^{1-\frac{1}{A}}\exp(\log^{\frac{2}{3}+\varepsilon_2}x),x]$ we have
    \begin{equation*}
        \Delta(x,h,\frac{a}{q})\ll h \exp(-\log^{\varepsilon_1}x).
    \end{equation*}
\end{lemma}
\begin{proof}
    We apply \cite[Lemmas 3.3 \& 3.4]{Har26} to \cite[Corollary 2.6]{Har26}.
\end{proof}
\begin{lemma}\label{lemma: all GRH}
    Assume \nameref{con: GRH}, let $q\in\N$ and let $h\in[\sqrt{qx}\log^{2+\varepsilon}qx,x]$ for any $\varepsilon>0$. We have 
    \begin{equation*}
        \Delta(x,h,\frac{a}{q})\ll \sqrt{qx}\log^{2}qx.
    \end{equation*}
    
\end{lemma}
\begin{proof}
    We apply \cite[Lemmas 3.3 \& 3.4]{Har26} to \cite[Corollary 2.3]{Har26}.
\end{proof}
Due to a result by Gallagher \cite{Gal70} we may also reduce the mean squared squared error term of $\Delta$ to error terms on rational points.
\begin{lemma}\label{lemma: Gallagher}
    Let $1\leq a \leq q \leq Q$ with $(a,q)=1$. We have
    \begin{equation*}
        \int_{\fM(q,a)}|\Delta(x,h,\alpha)|^2\diff\alpha\ll \frac{1}{(qQ)^2}\int_{x-h}^x|\Delta(t,\frac{qQ}{2},\frac{a}{q})|^2\diff t.
    \end{equation*}
\end{lemma}
\begin{proof}
    For $1\leq a \leq q \leq Q$ with $(a,q)=1$, we define
    \begin{equation*}
        c(n,q,a):=\begin{cases}
            \log n \thinspace\e(na/q)-\frac{\mu(q)}{\varphi(q)} & \text{if } x-h<n\leq x \text{ is a prime,}\\
            -\frac{\mu(q)}{\varphi(q)} & \text{if } x-h<n\leq x \text{ is not a prime,}\\
            0 & \text{otherwise.}
    \end{cases}
    \end{equation*}
    By Gallagher's Lemma (see \cite[Lemma 1]{Gal70}), we have
    \begin{equation*}
    \begin{split}
        \int_{\fM(q,a)}|\Delta(x,h,\alpha)|^2\diff\alpha=&\int_{-\frac{1}{qQ}}^{\frac{1}{qQ}}\bigg|\sum_{x-h<n\leq x}c(n,q,a)\e(n\beta)\bigg|^2\diff\beta\\
        \ll&\frac{4}{(qQ)^2}\int_{x-h}^{x}\bigg|\sum_{t-qQ/2<n\leq t}c(n,q,a)\bigg|^2\diff t\\
        \ll& \frac{1}{(qQ)^2}\int_{x-h}^x|\Delta(t,\frac{qQ}{2},\frac{a}{q})|^2\diff t.
    \end{split}
    \end{equation*}

\end{proof}
The following two Lemmas give estimates for the average squared error terms $\Delta$ at rational points.
\begin{lemma}\label{lemma: almost all zero density}
    Assume Condition \ref{con: zero density}. Let $0<\varepsilon_1<\varepsilon_2< 1/3$ and let $q\leq \log^\ell x$ for any $\ell>0$. For any $h\in [x^{1-\frac{2}{A}}\exp(\log^{\frac{2}{3}+\varepsilon_2}x),x]$ we have
    \begin{equation*}
        \psum_{a\thinspace (q)}\int_X^{2X}|\Delta(x,h,a/q)|^2\diff x\ll h^2 X\exp(-\log^{\varepsilon_1}x).
    \end{equation*}
\end{lemma}
\begin{proof}
    We apply \cite[Lemmas 3.3 \& 3.5]{Har26} to \cite[Corollary 2.13]{Har26}.
\end{proof}
\begin{lemma}\label{lemma: almost all GRH}
    Assume \nameref{con: GRH}, let $q\in\N$ and let $h\in[q\log^{2+\varepsilon} qx,x]$ for any $\varepsilon>0$. We have
    \begin{equation*}
        \psum_{a\thinspace (q)}\int_X^{2X}|\Delta(x,h,a/q)|^2\diff x\ll qhX\log^2 qx.
    \end{equation*}
\end{lemma}
\begin{proof}
    We apply \cite[Lemmas 3.3 \& 3.5]{Har26} to \cite[Corollary 2.10]{Har26}.
\end{proof}
We may obtain an upper bound for the generating function on the minor arcs via Weyl's inequality.
\begin{lemma}\label{lemma: Weyls inequality}
    Let $1\leq a\leq q \leq Q$ and let $\alpha\in\fM(q,a)$. We have
    \begin{equation*}
        S(x,x,\alpha)\ll \bigg(\sqrt{qx}+\frac{x}{\sqrt{q}}+x^\frac{4}{5}\bigg)\log^3 x.
    \end{equation*}
\end{lemma}
\begin{proof}
    See e.g. \cite[Theorem 13.6]{IK04}.
\end{proof}

The following Lemma is a version of Parseval's theorem.
\begin{lemma}\label{lemma: Parseval}
    Let $\mathfrak{I}\subset \R/\Z$ and let $S(x,h,1)\ll h$. We have
    \begin{equation*}
        \int_\mathfrak{I} |S(x,h,\alpha)|^2 \diff \alpha\ll h\log x.
    \end{equation*}
    If we assume Condition \ref{con: zero density}, we may take $h\in[x^{1-\frac{1}{A}}\exp(\log^{\frac{2}{3}+\varepsilon}x),x]$ for any $\varepsilon>0$. If we assume \nameref{con: GRH}, we may take $h\in[\sqrt{x}\log^{2}x,x]$ instead.
\end{lemma}
\begin{proof} 
    By non-negativity and by exchanging the order of integration and summation we have
    \begin{equation*}
    \begin{split}
        \int_\mathfrak{I} |S(x,h,\alpha)|^2 \diff \alpha \leq & \int_0^1 |S(x,h,\alpha)|^2 \diff \alpha\\
        = & \int_0^1 \sum_{x-h<p_1,p_2\leq x}\log p_1 \log p_2 \e((p_1-p_2)\alpha) \diff \alpha\\
        = & \sum_{x-h<p_1,p_2\leq x}\log p_1 \log p_2 \int_0^1\e((p_1-p_2)\alpha) \diff \alpha.
    \end{split}
    \end{equation*}
    The integral on the right-hand side is $1$ if $p_1=p_2$, and is $0$ otherwise. Therefore we have
    \begin{equation*}
    \begin{split}
        \int_\mathfrak{I} |S(x,h,\alpha)|^2 \diff \alpha \leq & \sum_{x-h<p\leq x}\log^2 p\leq S(x,h,1)\log x\ll h\log x.
    \end{split}
    \end{equation*}
    The ranges for $h$ under Condition \ref{con: zero density} \& \ref{con: GRH} follow immediately from Lemmas \ref{lemma: all zero density} \& \ref{lemma: all GRH}, respectively.
\end{proof}

Furthermore we give an upper bound for the geometric series $v(x,h,\beta)$.

\begin{lemma}\label{lemma: geometric bound}
    For $\beta\in \R$, we have
    \begin{equation*}
        |v(x,h,\beta)| \ll \min (h,|\beta|^{-1}).
    \end{equation*}
\end{lemma}
\begin{proof}
    Since $v(x,h,\beta)$ is a geometric series, we have
    \begin{equation*}
        |v(x,h,\beta)| \leq \min (h,|1-\e(\beta)|^{-1})\ll \min (x,|\beta|^{-1}).
    \end{equation*}
\end{proof}
\begin{lemma}\label{lemma: totient lower bound}
    Let $q\in\N$. For any $\varepsilon>0$ we have
    \begin{equation*}
        q^{1-\varepsilon}\ll \varphi(q) \leq q.
    \end{equation*}
\end{lemma}
\begin{proof}
    See e.g. \cite[Theorem 327]{HW38}.
\end{proof}
\subsection{Singular series}

In this subsection we are concerned with the truncated Ramanujan sum
\begin{equation}
    \mathfrak{S}(N,P):=\sum_{q\leq P}\frac{\mu(q)}{\varphi(q)^3}\psum_{a(q)} \eq(-Na).
\end{equation}
The first Lemma extends the second argument of $\mathfrak{S}$ to infinity, thereby incurring a small error term.
\begin{lemma}\label{lemma: Ramanujan sum error term}
    For any $\varepsilon>0$ we have
    \begin{equation*}
        \mathfrak{S}(N,P)=\mathfrak{S}(N,\infty)+\mathcal{O}(P^{-1+\varepsilon}).
    \end{equation*}
\end{lemma}
\begin{proof}
By Lemma \ref{lemma: totient lower bound}, we have
\begin{equation*}
\begin{split}
    \mathfrak{S}(N,P)=&\sum_{q=1}^\infty\frac{\mu(q)}{\varphi(q)^3}\psum_{a(q)} \eq(-Na)+\mathcal{O}\bigg(\sum_{q>P}\frac{1}{\varphi(q)^2}\bigg)\\
    =&\mathfrak{S}(N,\infty)+\mathcal{O}\bigg(\int_{P}^\infty t^{-2+\varepsilon}\diff t\bigg)\\
    =&\mathfrak{S}(N,\infty)+\mathcal{O}(P^{-1+\varepsilon}).
\end{split}
\end{equation*}
\end{proof}

\begin{lemma}\label{lemma: Ramanujan sum multiplicative}
    For $N\in \N$, we have
    \begin{equation*}
        \mathfrak{S}(N,\infty)=\mathfrak{S}(N).
    \end{equation*}
\end{lemma}
\begin{proof}
    Note that
    \begin{equation*}
        f(q):=\psum_{a(q)}\eq(-Na)
    \end{equation*}
    is multiplicative, i.e. $f(qr)=f(q)f(r)$ whenever $(q,r)=1$. Note that $\mu$ and $\varphi$ are also multiplicative. We have
    \begin{equation*}
    \begin{split}
        \mathfrak{S}(N,\infty)=& \prod_p\bigg(1+ \frac{\mu(p)f(p)}{\varphi^3(p)} + \frac{\mu(p^2)f(p^2)}{\varphi^3(p^2)}+\dots \bigg)\\
        =&\prod_p\bigg(1- \frac{f(p)}{\varphi^3(p)}\bigg).
    \end{split}
    \end{equation*}
    Furthermore we note that $\varphi(p)=p-1$ and
    \begin{equation*}
        f(p)=\sum_{a=1}^{p-1} \ep(-Na)=\begin{cases}
            p-1 & \text{if } p\mid N\\
            -1 & \text{if } p\nmid N,
        \end{cases}
    \end{equation*}
    thus concluding the proof.
\end{proof}

\section{Proof of Theorem \ref{thm: main zero density}}
In this section we prove Theorem \ref{thm: main zero density}. We assume Condition \ref{con: zero density} throughout the section.
\subsection{The minor arcs under zero-density estimates}
In this section we study the integral $R_{\fm}$ on the minor arcs $\fm$ under the assumption of Condition \ref{con: zero density}. For zero-density estimates it suffices to use Weyl's inequality (Lemma \ref{lemma: Weyls inequality}) to obtain a sufficiently strong result.

\begin{lemma}\label{lemma: minor arcs density}
    Assume Condition \ref{con: zero density}. Let $y\in[N^{1-\frac{1}{A}}\exp(\log^{\frac{2}{3}+\varepsilon_1}N),N/2]$ for any $\varepsilon_1>0$. Let $8< 8+2\varepsilon_2 \leq c_1<c_2$ and let $U\in [\log^{k}N,N]$ for any $k>c_1+c_2$. We have
    \begin{equation*}
        R_\fm\ll\frac{Uy}{\log^{\varepsilon_2}N}.
    \end{equation*}
\end{lemma}
\begin{proof}
    By the Hölder inequality and Lemmas \ref{lemma: Weyls inequality} \& \ref{lemma: Parseval} we have
    \begin{equation*}\label{eq: minor arcs 1}
    \begin{split}
        R_{\fm}\ll &\sup_{\alpha\in\fm}|S_3(\alpha)|\bigg(\int_\fm |S_1(\alpha)|^2\diff \alpha \int_\fm | S_2(\alpha)|^2\diff \alpha\bigg)^{1/2}\\
        \ll& y\log N \sup_{\alpha\in\fm'}|S_3(\alpha)|\\
        \ll& Uy  (\log^{-\frac{c_2}{2}}N+\log^{-\frac{c_1}{2}}N+U^{-\frac{1}{5}})\log N\log^3 U\\
        \ll& Uy \log^{4-\frac{c_1}{2}}N
        \ll\frac{Uy}{\log^{\varepsilon_2}N}.
    \end{split}
    \end{equation*}
\end{proof}
Lemma \ref{lemma: minor arcs density} implies that the minor arc contribution under zero-density estimates is $o(Uy)$ for $U\in [\log^{16+\varepsilon}N,y]$ for any $\varepsilon>0$ and an appropriate choice of $y$.

\subsection{The major arcs under zero-density estimates}
In this section we study the integral $R_\fM$ on the major arcs $\fM$ under the assumption of Condition \ref{con: zero density}. The first Lemma of this section calculates the first two error terms of the standard circle method.

\begin{lemma}\label{lemma: major arcs density 1}
    Assume Condition \ref{con: zero density}. For any $0<\varepsilon_1<\frac{1}{3}$, let $y\in[N^{1-\frac{1}{A}}\exp(\log^{\frac{2}{3}+\varepsilon_1} N),N/2]$ and for any $0< \varepsilon_0<\varepsilon_2<\frac{1}{3}$ let $U\in [y^{1-\frac{2}{A}}\exp(\log^{\frac{2}{3}+\varepsilon_2}y),y]$. We have 
    \begin{equation*}
        R_\fM =\sum_{q\leq P}\frac{\mu^2(q)}{\varphi^2(q)}\psum_{a(q)} \int_{-\frac{1}{qQ}}^{\frac{1}{qQ}} S_1(\frac{a}{q}+\beta) v_2(\beta) v_3(\beta) \e(-N(\frac{a}{q}+\beta))\diff\beta + \cO\bigg(\frac{Uy}{\exp(\log^{\varepsilon_0}N)}\bigg).
    \end{equation*}
\end{lemma}
\begin{proof}

By Lemma \ref{lemma: disjoint major arcs} we may split the integral over $\fM$ into a sum of integrals over disjoint arcs $\fM(q,a)$ with $1\leq a\leq q \leq P=\log^{c_1}N$ and $(a,q)=1$. We obtain
\begin{equation}\label{eq: major arcs density I1 I2}
\begin{split}
    R_\fM =& \sum_{q\leq P}\psum_{a(q)} \int_{-\frac{1}{qQ}}^{\frac{1}{qQ}} S_1(\frac{a}{q}+\beta) S_2(\frac{a}{q}+\beta) S_3(\frac{a}{q}+\beta) \e(-N(\frac{a}{q}+\beta))\diff\beta\\
    =& \sum_{q\leq P}\frac{\mu^2(q)}{\varphi^2(q)}\psum_{a(q)} \int_{-\frac{1}{qQ}}^{\frac{1}{qQ}} S_1(\frac{a}{q}+\beta) v_2(\beta) v_3(\beta) \e(-N(\frac{a}{q}+\beta))\diff\beta + I_1+I_2,
\end{split}
\end{equation}
where
\begin{equation*}
    I_1:=\int_\fM S_1(\alpha) S_2(\alpha) \Delta_3(\alpha) \e(-N\alpha)\diff\alpha,
\end{equation*}
and
\begin{equation*}
    I_2:=\int_\fM S_1(\alpha) \Delta_2(\alpha) (S_3(\alpha)-\Delta_3(\alpha)) \e(-N\alpha)\diff\alpha.
\end{equation*}
By H\"older's inequality and Lemmas \ref{Lemma: Generating function} \& \ref{lemma: Parseval} we have
\begin{equation}\label{eq: major arcs density I1}
\begin{split}
    I_1\leq&\sup_{\alpha\in\fM}|\Delta_3(\alpha)| \bigg(\int_\fM |S_1(\alpha)|^2\diff\alpha \int_\fM|S_2(\alpha)|^2 \diff\alpha\bigg)^\frac{1}{2}\\
    \ll& y\log N \max_{\alpha\in\fM}|\Delta_3(\alpha)| \\
    \ll& y\log N \max_{q\leq P}\frac{U}{qQ}\max_{(a,q)=1} \max_{0<t\leq U} |\Delta(U,t,\frac{a}{q})|
\end{split}
\end{equation}
for $y\in[N^{1-\frac{1}{A}}\exp(\log^{\frac{2}{3}+\varepsilon_1} N),N/2]$. By applying Lemma \ref{lemma: all unconditional} and the lower bound \mbox{$U\geq \exp(\log^\frac{2}{3}N)$} we obtain 
\begin{equation*}
    I_1\ll Uy\exp(-\log^{\varepsilon_3}U)\ll Uy\exp(-\log^{\varepsilon_0}N)
\end{equation*}
for any $0<\varepsilon_0<\frac{1}{3}$ and $\frac{1}{2}<\varepsilon_3<\frac{2}{5}$. For $I_2$ we apply H\"older's inequality and Lemmas \ref{lemma: Gallagher}, \ref{lemma: Parseval} \& \ref{lemma: geometric bound} to obtain

\begin{equation}\label{eq: major arcs density I2}
\begin{split}
    I_2\leq& \bigg(\int_\fM |S_1(\alpha)|^2\diff \alpha \int_\fM |\Delta_2(\alpha)(S_3(\alpha)-\Delta_3(\alpha))|^2 \diff\alpha\bigg)^\frac{1}{2}\\
    \ll& \sqrt{y\log N} \bigg(\sum_{q\leq P}\psum_{a\thinspace(q)}\int_{-\frac{1}{qQ}}^{\frac{1}{qQ}} \bigg|\Delta_2(\frac{a}{q}+\beta)\frac{v_3(\beta)}{\varphi(q)}\bigg|^2 \diff\beta\bigg)^\frac{1}{2}\\
    \ll& U\sqrt{y\log N} \bigg(\sum_{q\leq P}\frac{1}{(\varphi(q)qQ)^2}\psum_{a\thinspace(q)}\int_y^{2y} |\Delta(t,\frac{qQ}{2},\frac{a}{q})|^2 \diff t\bigg)^\frac{1}{2}
\end{split}
\end{equation}
for $y\in[N^{1-\frac{1}{A}}\exp(\log^{\frac{2}{3}+\varepsilon_1} N),N/2]$. By applying Lemmas \ref{lemma: almost all zero density} \& \ref{lemma: totient lower bound} we obtain
\begin{equation*}
\begin{split}
    I_2\ll&Uy\sqrt{\log N} \exp(-\frac{1}{2}\log^{\varepsilon_4}y)\bigg(\sum_{q\leq P}\frac{1}{\varphi^2(q)}\bigg)^\frac{1}{2}\\
    \ll& Uy\exp(-\log^{\varepsilon_0}N),
\end{split}
\end{equation*}
for any $0<\varepsilon_0<\varepsilon_4<\varepsilon_2<\frac{1}{3}$. Note that this is independent of the choices of $P=\log^{c_1}N$ and $Q=U/\log^{c_2}N$. 
\end{proof}
Now, we are unfortunately unable to calculate the third error term of the circle method as one usually would. By Gallagher's inequality \cite{Gal70} we would obtain an error term of the form 
\begin{equation*}
    \int_{N-2y}^{N-y}|\Delta(t,\frac{qQ}{2},\frac{a}{q})|^2\diff t,
\end{equation*}
which we are unable to estimate for small $y$ (compared to $N$). Therefore, in the following Lemma we reduce the ternary Goldbach problem on the major arcs to the binary Goldbach problem, thereby incurring an error term of $\cO(U^2)$. This is also the reason we require \mbox{$U=o(y)$.}
\begin{lemma}\label{lemma: major arcs density 2}
    Assume Condition \ref{con: zero density}. Let $0<\varepsilon_0<\varepsilon_1$ and let $0<\varepsilon_2<\frac{1}{3}$. Let $y\in [N^{1-\frac{1}{A}}\exp(\log^{\frac{2}{3}+\varepsilon_2}N),N/2]$ and $U\in[y^{1-\frac{1}{A}}\exp(\log^{\frac{2}{3}+\varepsilon_2}y),y/\log^{\varepsilon_1}N]$. We have
    \begin{equation*}
        R_\fM =\mathfrak{S}(N)Uy +\cO\bigg(\frac{Uy}{\log^{c_1-\varepsilon_3}N} +\frac{Uy}{\log^{\varepsilon_0}N}+Uy(\log N)^{\frac{1+3(c_1-c_2)}{2}} \bigg),
    \end{equation*}
    for any $0<\varepsilon_0<c_1<c_2-\frac{1}{3}$.
\end{lemma}
\begin{proof}
    We take Lemma \ref{lemma: major arcs density 1} and extend the integral in the main term to an entire unit interval, thereby incurring an error term. We may bound this error term by applying H\"older's inequality and Lemmas \ref{lemma: Parseval}, \ref{lemma: geometric bound} \& \ref{lemma: totient lower bound}.
    \begin{equation}\label{eq: major arcs density extension}
    \begin{split}
        R_\fM=&\sum_{q\leq P}\frac{\mu^2(q)}{\varphi^2(q)}\psum_{a(q)} \int_{-\frac{1}{2}}^{\frac{1}{2}} S_1(\frac{a}{q}+\beta) v_2(\beta) v_3(\beta) \e(-N(\frac{a}{q}+\beta))\diff\beta +\cO\bigg(\frac{Uy}{\exp(\log^{\varepsilon_4} N)}\bigg)\\
        &+\cO\bigg(\sum_{q\leq P}\frac{1}{\varphi^2(q)}\psum_{a(q)} \bigg(\int_{\frac{1}{qQ}}^{\frac{1}{2}} |v_2(\beta) v_3(\beta)|^2  \diff\beta \int_{\frac{1}{qQ}}^{\frac{1}{2}} |S_1(\frac{a}{q}+\beta)|^2  \diff\beta\bigg)^\frac{1}{2} \bigg)\\
        =&\sum_{q\leq P}\frac{\mu^2(q)}{\varphi^2(q)}\psum_{a(q)} \int_{-\frac{1}{2}}^{\frac{1}{2}} S_1(\frac{a}{q}+\beta) v_2(\beta) v_3(\beta) \e(-N(\frac{a}{q}+\beta))\diff\beta +\cO\bigg(\frac{Uy}{\exp(\log^{\varepsilon_4} N)}\bigg)\\
        &+\cO(U^\frac{3}{2}y^\frac{1}{2}(\log N)^{\frac{1+(3+\varepsilon_5)c_1-3c_2}{2}}),
    \end{split}
    \end{equation}
    for any $0<\varepsilon_4<\varepsilon_2$ and $\varepsilon_5>0$. Now let $1\leq a \leq q \leq P$ with $(a,q)=1$. We take the integral in the main term of \eqref{eq: major arcs density extension} and exchange the order of integration and summation.
    \begin{equation*}
    \begin{split}
        I_3(q,a):=&\int_{-\frac{1}{2}}^{\frac{1}{2}} S_1(\frac{a}{q}+\beta) v_2(\beta) v_3(\beta) \e(-N(\frac{a}{q}+\beta))\diff\beta\\
        =&\sum_{N-2y<p_1\leq N-y}\log p_1 \e((p_1-N)\frac{a}{q})\sum_{\substack{y<n_2\leq 2y\\n_3\leq U}}\int_{-\frac{1}{2}}^{\frac{1}{2}} \e((p_1+n_2+n_3-N)\beta)\diff\beta.
    \end{split}
    \end{equation*}
    By Lemma \ref{lemma: orth relation}, the double sum over $n_2,n_3$ becomes the number of positive integer solutions to the equation $n_2+n_3=N-p_1$, where $y< n_2\leq 2y$, $n_3\leq U$ and $p_1$ is fixed. The solution to this combinatorial problem is $\min(U,N-y-p_1)$. Therefore we have
    \begin{equation}\label{eq: major arcs density I3}
    \begin{split}
        I_3(q,a)=&\sum_{N-2y<p_1\leq N-y}\log p_1 \e((p_1-N)\frac{a}{q})\min(U,N-y-p_1)\\
        =&U\e(-Na/q)S_1(\frac{a}{q})+\cO(U^2)\\
        =&Uy\frac{\mu(q)}{\varphi(q)}\e(-Na/q)+\cO(U\Delta_1(\frac{a}{q}) +U^2).
    \end{split}
    \end{equation}
    We combine \eqref{eq: major arcs density extension} \& \eqref{eq: major arcs density I3} and apply Lemma \ref{lemma: all zero density} for $y\in [N^{1-\frac{1}{A}}\exp(\log^{\frac{2}{3}+\varepsilon_2}N),N/2]$ as well as Lemma \ref{lemma: Ramanujan sum error term} to extend the sum over $q$ to infinity.
    \begin{equation*}
    \begin{split}
        R_\fM=&Uy\sum_{q\leq P}\frac{\mu^3(q)}{\varphi^3(q)}\psum_{a(q)}\e(-Na/q)+\cO\bigg(\bigg(\frac{Uy}{\exp(\log^{\varepsilon_4}N)} +U^2\bigg)\sum_{q\leq P}\frac{1}{\varphi(q)}\bigg)\\
        &+\cO\bigg(\frac{Uy}{\exp(\log^{\varepsilon_4} N)}+U^\frac{3}{2}y^\frac{1}{2}(\log N)^{\frac{1+(3+\varepsilon_5)c_1-3c_2}{2}}\bigg)\\
        =& \mathfrak{S}(N,\infty)Uy +\cO\bigg(\frac{Uy}{\log^{\varepsilon_0}N}+Uy(\log N)^{\frac{1+3(c_1-c_2)}{2}} \bigg).
    \end{split}
    \end{equation*}
    Lastly we apply Lemma \ref{lemma: Ramanujan sum multiplicative} to complete the proof.
\end{proof}
\subsection{The weighted asymptotic formula under zero-density estimates}
In this subsection we combine the major and minor arcs to obtain a weighted asymptotic formula for $R(N,y,U)$ given a zero-density estimate. Theorem \ref{thm: main zero density} follows from Lemma \ref{lemma: weighted reps} and the following Lemma.
\begin{lemma}
    Assume Condition \ref{con: zero density}. Let $0<\varepsilon_0<\varepsilon_1$ and let $0<\varepsilon_2<\frac{1}{3}$. For any $y\in [N^{1-\frac{1}{A}}\exp(\log^{\frac{2}{3}+\varepsilon_2}N),N/2]$ and $U\in[y^{1-\frac{1}{A}}\exp(\log^{\frac{2}{3}+\varepsilon_2}y),y/\log^{\varepsilon_1}N]$. We have
    \begin{equation*}
        R(N,y,U) =\mathfrak{S}(N)Uy +\cO\bigg(\frac{Uy}{\log^{\varepsilon_0}N}\bigg).
    \end{equation*}
\end{lemma}
\begin{proof}
    The proof follows directly from Lemmas \ref{lemma: minor arcs density} \& \ref{lemma: major arcs density 2} by taking any \mbox{$8<8+2\varepsilon_0< c_1<c_2-\frac{1}{3}$.}
\end{proof}

\section{Proof of Theorem \ref{thm: main GRH}}
In this section we prove Theorem \ref{thm: main GRH}. We assume Condition \nameref{con: GRH} throughout the section.
\subsection{The minor arcs under \nameref{con: GRH}}
In this section we study the integral $R_{\fm}$ on the minor arcs under the assumption of \nameref{con: GRH}. In contrast to the zero-density estimate case, we are able to maintain a smaller $P$ (and therefore a smaller contribution from the error terms of the major arcs) by directly calculating an upper bound via Lemma \ref{lemma: all GRH}. This is possible since in the \nameref{con: GRH} case we are not restricted by a logarithmic upper bound on the modulus $q\leq \log^\ell U$ for some $\ell>0$.

\begin{lemma}\label{lemma: minor arcs GRH}
    Assume \nameref{con: GRH}. Let $y\in[\sqrt{N}\log^{2} N,N/2]$. Let $1+\varepsilon<c_1<c_2$ and let $U\in[\log^k N,y]$ with $k>1+2c_2+\varepsilon$ for any $\varepsilon>0$. We have
    \begin{equation*}
        R_\fm\ll \frac{Uy}{\log^{\varepsilon}N}.
    \end{equation*}
\end{lemma}
\begin{proof}
By the Hölder inequality and Lemmas \ref{Lemma: Generating function}, \ref{lemma: all GRH}, \ref{lemma: Parseval} \& \ref{lemma: geometric bound} we have
\begin{equation*}
\begin{split}
    R_{\fm}\ll &\sup_{\alpha\in\fm}|S_3(\alpha)|\bigg(\int_\fm |S_1(\alpha)|^2\diff \alpha \int_\fm | S_2(\alpha)|^2\diff \alpha\bigg)^{1/2}\\
    \ll & y\log N \max_{P<q\leq Q}\max_{(a,q)=1}\bigg(\frac{U}{\varphi(q)}+ \bigg(1+\frac{U}{qQ}\bigg) \max_{u\leq U}|\Delta(U,u,\frac{a}{q})|\bigg)\\
    \ll & y\log N \max_{P<q\leq Q}\bigg(\frac{U}{\varphi(q)}+ \frac{\sqrt{U}}{\sqrt{q}} \log^{c_2} N\log^2 U\bigg).
\end{split}
\end{equation*}
By Lemma \ref{lemma: totient lower bound} we then obtain
\begin{equation*}
\begin{split}
    R_{\fm}\ll & Uy \log^{1-c_1+\varepsilon_1} N + U^{\frac{1}{2}} y \log^{1+c_2-\frac{c_1}{2}} N \log^2 U\\
    \ll&\frac{Uy}{\log^{\varepsilon}N},
\end{split}
\end{equation*}
for any $0<\varepsilon_1<c_1-1-\varepsilon$.
\end{proof}
Lemma \ref{lemma: minor arcs GRH} implies that the minor arc contribution under \nameref{con: GRH} is $o(Uy)$ for $U\in [\log^{3+\varepsilon}N,y]$ for any $\varepsilon>0$ and an appropriate choice of $y$.

\subsection{The major arcs under \nameref{con: GRH}}
In this section we study the integral $R_\fM$ on the major arcs $\fM$ under the assumption of \nameref{con: GRH}.

\begin{lemma}\label{lemma: major arcs 1 GRH}
    Assume \nameref{con: GRH}. Let $y\in[\sqrt{N}\log^{2}N,N/2]$ and let $U\in[\log^k N,y]$ for any \mbox{$k>2\varepsilon+\max\{2+2c_2,3+c_2\}$} we obtain
    \begin{equation*}
        R_\fM =\sum_{q\leq P}\frac{\mu^2(q)}{\varphi^2(q)}\psum_{a(q)} \int_{-\frac{1}{qQ}}^{\frac{1}{qQ}} S_1(\frac{a}{q}+\beta) v_2(\beta) v_3(\beta) \e(-N(\frac{a}{q}+\beta))\diff\beta + \cO\bigg(\frac{Uy}{\log^{\varepsilon}N}\bigg).
    \end{equation*}
\end{lemma}
\begin{proof}
Similarly to Lemma \ref{eq: major arcs density I1 I2}, we obtain Equations \eqref{eq: major arcs density I1 I2}, \eqref{eq: major arcs density I1} \& \eqref{eq: major arcs density I2}, except, since we use the \nameref{con: GRH} result of Lemma \ref{lemma: Parseval}, these equations hold for $y\in[\sqrt{N}\log^{2}N,N/2]$. By applying Lemmas \ref{lemma: all GRH}, \ref{lemma: almost all GRH} \& \ref{lemma: totient lower bound} we then obtain

\begin{equation*}
\begin{split}
    I_1+I_2\ll& \sqrt{U}y\log^{1+c_2} N \log^2 U + Uy \bigg(\frac{\log^{3+c_2} N}{U}\sum_{q\leq P}\frac{1}{\varphi^2(q)}\bigg)^\frac{1}{2}\\
    \ll& Uy\log^{1+c_2-\frac{k}{2}} N \log^2 U+ Uy\log^\frac{3+c_2-k}{2} N.
\end{split}
\end{equation*}
Taking $k>2\varepsilon+\max\{2+2c_2,3+c_2\}$ concludes the proof.

\end{proof}

\begin{lemma}\label{lemma: major arcs GRH 2}
    Assume Condition \ref{con: zero density}. Let $0<\varepsilon_0<\varepsilon_1,c_1$. Let $y\in [\sqrt{N}\log^{2+\varepsilon_1}N,N/2]$ and $U\in[\log^{k}N,y/\log^{\varepsilon_1}N]$ for any $k>2\varepsilon_0+\max\{2+2c_2,3+c_2\}$. We have
    \begin{equation*}
        R_\fM =\mathfrak{S}(N)Uy +\cO\bigg(\frac{Uy}{\log^{\varepsilon_0}N}+\sqrt{U^3y(\log N)^{1+(3+\varepsilon_2)c_1-3c_2}} \bigg),
    \end{equation*}
    for any $\varepsilon_2>0$.
\end{lemma}
\begin{proof}
    Similarly to \eqref{eq: major arcs density extension}, we take Lemma \ref{lemma: major arcs 1 GRH} and extend the integral in the main term to an entire unit interval, thereby incurring an error term. We may bound this error term by applying H\"older's inequality and Lemmas \ref{lemma: Parseval}, \ref{lemma: geometric bound} \& \ref{lemma: totient lower bound}.
    \begin{equation}\label{eq: major arcs GRH extension}
    \begin{split}
        R_\fM=&\sum_{q\leq P}\frac{\mu^2(q)}{\varphi^2(q)}\psum_{a(q)} \int_{-\frac{1}{2}}^{\frac{1}{2}} S_1(\frac{a}{q}+\beta) v_2(\beta) v_3(\beta) \e(-N(\frac{a}{q}+\beta))\diff\beta +\cO\bigg(\frac{Uy}{\log^{\varepsilon_0} N}\bigg)\\
        &+\cO\bigg(\sum_{q\leq P}\frac{1}{\varphi^2(q)}\psum_{a(q)} \bigg(\int_{\frac{1}{qQ}}^{\frac{1}{2}} |v_2(\beta) v_3(\beta)|^2  \diff\beta \int_{\frac{1}{qQ}}^{\frac{1}{2}} |S_1(\frac{a}{q}+\beta)|^2  \diff\beta\bigg)^\frac{1}{2} \bigg)\\
        =&\sum_{q\leq P}\frac{\mu^2(q)}{\varphi^2(q)}\psum_{a(q)} \int_{-\frac{1}{2}}^{\frac{1}{2}} S_1(\frac{a}{q}+\beta) v_2(\beta) v_3(\beta) \e(-N(\frac{a}{q}+\beta))\diff\beta +\cO\bigg(\frac{Uy}{\log^{\varepsilon_0} N}\bigg)\\
        &+\cO(\sqrt{U^3y(\log N)^{1+(3+\varepsilon_5)c_1-3c_2}}),
    \end{split}
    \end{equation}
    for any $\varepsilon_0,\varepsilon_2>0$. Next we combine \eqref{eq: major arcs density I3} \& \eqref{eq: major arcs GRH extension} and apply Lemma \ref{lemma: all GRH} for $y\in [\sqrt{N}\log^{2+\varepsilon_1}N,N/2]$ as well as Lemma \ref{lemma: Ramanujan sum error term} to extend the sum over $q$ to infinity.
    \begin{equation*}
    \begin{split}
        R_\fM=&Uy\sum_{q\leq P}\frac{\mu^3(q)}{\varphi^3(q)}\psum_{a(q)}\e(-Na/q)+\cO\bigg(\bigg(\frac{Uy}{\log^{\varepsilon_1}N} +U^2\bigg)\sum_{q\leq P}\frac{1}{\varphi(q)}\bigg)\\
        &+\cO\bigg(\frac{Uy}{\log^{\varepsilon_0} N}+\sqrt{U^3y(\log N)^{1+(3+\varepsilon_5)c_1-3c_2}}\bigg)\\
        =& \mathfrak{S}(N,\infty)Uy +\cO\bigg(\frac{Uy}{\log^{\varepsilon_0}N}+\sqrt{U^3y(\log N)^{1+(3+\varepsilon_5)c_1-3c_2}} \bigg).
    \end{split}
    \end{equation*}
    Lastly we apply Lemma \ref{lemma: Ramanujan sum multiplicative} to complete the proof.
\end{proof}
\subsection{The weighted asymptotic formula under \nameref{con: GRH}}
In this subsection we combine the major and minor arcs to obtain a weighted asymptotic formula for $R(N,y,U)$ given \nameref{con: GRH}. Theorem \ref{thm: main zero density} follows from Lemma \ref{lemma: weighted reps} and the following Lemma.
\begin{lemma}
    Assume \nameref{con: GRH}. Let $0<4\varepsilon_0<\varepsilon_1$. Let $y\in [\sqrt{N}\log^{2+\varepsilon_1},N/2]$ and\\ \mbox{$U\in[\log^{4+\varepsilon_1}N,y/\log^{\varepsilon_1}N]$.} We have
    \begin{equation*}
        R(N,y,U) =\mathfrak{S}(N)Uy +\cO\bigg(\frac{Uy}{\log^{\varepsilon_0}N}\bigg).
    \end{equation*}
\end{lemma}
\begin{proof}
    The proof follows directly from Lemmas \ref{lemma: minor arcs density} \& \ref{lemma: major arcs density 2}. For sufficiently small $U$ (e.g. $U\leq \sqrt{y}$) we take $1<1+\varepsilon_0<c_1<c_2<1+\frac{\varepsilon_1}{2}-\varepsilon_0$. For sufficiently large $U$ (e.g. $U> \sqrt{y}$) we take $1<1+\varepsilon_0<c_1<c_2-\frac{1}{3}$.
\end{proof}

\end{document}